\documentclass[11pt]{article}

\usepackage[margin=1.05in]{geometry}
\usepackage[T1]{fontenc}
\usepackage[utf8]{inputenc}
\usepackage{lmodern}
\usepackage{amsmath,amssymb,amsthm,mathtools}
\usepackage{booktabs}
\usepackage{graphicx}
\usepackage{microtype}
\usepackage{array}
\newcolumntype{L}[1]{>{\raggedright\arraybackslash}p{#1}}
\usepackage{enumitem}
\usepackage{xcolor}
\usepackage{hyperref}
\usepackage{url}

\hypersetup{
  colorlinks=true,
  linkcolor=blue!55!black,
  citecolor=blue!55!black,
  urlcolor=blue!55!black
}

\theoremstyle{plain}
\newtheorem{theorem}{Theorem}[section]
\newtheorem{proposition}[theorem]{Proposition}
\newtheorem{lemma}[theorem]{Lemma}
\newtheorem{corollary}[theorem]{Corollary}

\theoremstyle{definition}
\newtheorem{definition}[theorem]{Definition}

\newtheorem{finding}[theorem]{Computational finding}

\theoremstyle{remark}
\newtheorem{remark}[theorem]{Remark}

\newcommand{\Q}{\mathbb{Q}}
\newcommand{\Z}{\mathbb{Z}}

\newcommand{\dd}{\,d}
\newcommand{\floor}[1]{\left\lfloor #1\right\rfloor}
\newcommand{\ceil}[1]{\left\lceil #1\right\rceil}

\title{Tail Criteria, No-Go Audits, and Ap\'ery-Type Certificate Obstructions for the Irrationality of $e+\pi$}

\author{Runlong Yu\\
  The University of Alabama, Tuscaloosa, AL, USA\\
  \texttt{ryu5@ua.edu}}
\date{}

\begin{document}
\maketitle

\begin{abstract}
The irrationality of $e+\pi$ remains open, despite the separate transcendence of $e$ and $\pi$.  This paper studies the problem from the viewpoint of finite irrationality certificates and formulates a bounded \emph{no-go audit} for low-complexity Ap\'ery-type proof mechanisms.  The first part gives exact equivalences between the rationality hypothesis $e+\pi\in\Q$ and eventual arithmetic phenomena in factorial expansions: a ceiling recurrence, a factorial-Cantor digit condition, and a divisibility criterion.  These equivalences show precisely what rationality would force, but also explain why such tail criteria are not by themselves finite obstructions.  The second part formulates an Ap\'ery-type certificate framework based on integer linear forms
\[
        L_n=A_n(e+\pi)+B_n,\qquad A_n,B_n\in\Z,
        \qquad 0<|L_n|\to0.
\]
A mixed integration-by-parts identity produces a natural family of such forms from integer polynomials.  We then audit several low-complexity constructions, including mixed Pad\'e approximation, crossed separate approximations to $e$ and $\pi$, simple $J$-fractions, holonomic ansatzes, Rodrigues-type polynomial families, and an integer kernel-lattice search.  The main contribution is not an unsuccessful proof attempt, but a rigid boundary probe: a set of no-go filters that mark a forbidden zone where analytic smallness is destroyed by denominator clearing, coefficient growth, primitive reduction, or continued-fraction shadows.  In the final kernel-lattice audit, $145$ raw candidates reduce to $133$ primitive records; the degree-wise data show that the best signals are dominated by continued-fraction shadows, while non-CF candidates do not form a degree-continuing family.  Thus, within the tested low-complexity families, no non-circular Ap\'ery-type mechanism for $e+\pi$ is found.  The paper contributes rigorous tail criteria, a mixed-kernel certificate identity, and an explicit no-go audit framework for future attempts.
\end{abstract}

\section{Introduction}

The problem
\begin{equation}\label{eq:main-problem}
        e+\pi\notin\Q
\end{equation}
is a small-looking instance of a much larger difficulty in transcendental number theory \cite{Baker1975,Lang1966}.  The numbers $e$ and $\pi$ are separately transcendental, by the classical theorems of Hermite and Lindemann \cite{Hermite1874,Lindemann1882}, but their separate transcendence does not rule out a rational relation between them.  In particular, the sum of two transcendental numbers may be rational, and the known proofs of the transcendence of $e$ and of $\pi$ do not interact in a way that excludes cancellation in $e+\pi$.

This paper does not prove \eqref{eq:main-problem}.  Its purpose is more structural.  We isolate the arithmetic requirements that any elementary Ap\'ery-type proof would have to satisfy, in the tradition of integer-linear-form proofs such as Ap\'ery's theorem and Beukers' integral reformulation \cite{Apery1979,Beukers1979}, prove several exact criteria forced by the rationality hypothesis, and test a number of natural finite certificate mechanisms against those requirements.

The resulting contribution should be read as a bounded no-go theorem for a specified low-complexity universe, not as a universal impossibility theorem.  In this paper, a ``no-go'' statement means the following: a proposed mechanism is ruled out as an Ap\'ery-type certificate generator once its denominator-cleared or primitive integer linear forms fail to tend to zero, or once its strongest outputs reduce to continued-fraction shadows without a degree-continuing non-CF family.  This changes the interpretation of the computations.  They are not failed proof attempts; they delineate a tested safe zone and a tested forbidden zone for low-complexity Ap\'ery-type strategies.

The guiding point is that numerical approximation is not enough.  A sequence of rational approximations $B_n/A_n\to -(e+\pi)$ proves irrationality only when it arises from integer linear forms that are both nonzero and genuinely small:
\begin{equation}\label{eq:integer-linear-form-intro}
        L_n=A_n(e+\pi)+B_n,\qquad A_n,B_n\in\Z,
        \qquad 0<|L_n|\to0.
\end{equation}
If $e+\pi=a/b\in\Q$, then
\[
        bL_n=A_na+bB_n\in\Z.
\]
For every nonzero $L_n$ this gives $|bL_n|\ge1$, contradicting $L_n\to0$.  Thus the arithmetic size of the integer linear form, not merely the decimal closeness of $-B_n/A_n$, is the central object.

\subsection{Algebraic context}

The elementary algebraic statement closest to \eqref{eq:main-problem} is the following.

\begin{proposition}\label{prop:at-least-one}
At least one of $e+\pi$ and $e\pi$ is transcendental.
\end{proposition}

\begin{proof}
Let $S=e+\pi$ and $P=e\pi$.  If both $S$ and $P$ were algebraic, then $e$ and $\pi$ would be the roots of
\[
        X^2-SX+P=0.
\]
Hence both $e$ and $\pi$ would be algebraic, contradicting the transcendence of either one.  Therefore $S$ and $P$ cannot both be algebraic.
\end{proof}

This proposition does not identify which of $S$ and $P$ is transcendental.  The modern period viewpoint gives one broad language for such independence phenomena \cite{KontsevichZagier2001}.  Strong conjectures such as Schanuel's conjecture would imply much more; for instance, applying Schanuel's conjecture to $1$ and $i\pi$ would imply the algebraic independence of $e$ and $\pi$, and that independence would immediately imply the transcendence of $e+\pi$.  The present problem is therefore best viewed as a weak local shadow of a much deeper independence question.

\subsection{Main results and structure}

The paper has three mathematical components and one experimental no-go audit.

First, we prove exact equivalences between the hypothesis $e+\pi\in\Q$ and eventual behavior in factorial arithmetic.  In particular, if
\[
        A_N=\ceil{N!\pi},
\]
then $e+\pi\in\Q$ is equivalent to the eventual recurrence
\begin{equation}\label{eq:intro-ceiling-recurrence}
        A_{N+1}=(N+1)A_N-1.
\end{equation}
Equivalently, the factorial-Cantor digits of $\pi$ eventually satisfy $d_N(\pi)=N-2$, and an associated sequence $Q_N$ is eventually divisible by $N$.

Second, we formulate the finite certificate criterion \eqref{eq:integer-linear-form-intro}.  This separates genuine irrationality certificates from tail fingerprints and from ordinary rational approximation.

Third, we derive a mixed kernel identity.  For $P\in\Z[x]$, set
\begin{equation}\label{eq:SP-def-intro}
        S_P(x)=\sum_{j\ge0}(-1)^j P^{(j)}(x),
        \qquad A(P)=S_P(1).
\end{equation}
Only finitely many derivatives occur.  Then
\begin{equation}\label{eq:intro-kernel-identity}
        \int_0^1\left(P(x)e^x+\frac{4A(P)}{1+x^2}\right)\dd x
        =A(P)(e+\pi)-S_P(0).
\end{equation}
Thus integer polynomials produce integer linear forms in $e+\pi$.  The remaining problem is to make the integral small without losing arithmetic control.

Fourth, we treat the final bounded searches as a no-go audit.  The point is not merely that no certificate was found, but that the tested mechanisms fail for identifiable structural reasons.  In particular, the final kernel-lattice data visualize the competition between residual precision and denominator growth as the polynomial degree increases; see Figure~\ref{fig:round8-degree-audit}.

The computational part of the paper is deliberately stated as evidence inside bounded search spaces, not as a universal impossibility theorem.  Its conclusion is that the tested low-complexity families fail for a common reason: analytic cancellation can be achieved, but after denominator clearing, primitive reduction, or continued-fraction auditing, the resulting forms do not give a non-circular Ap\'ery-type sequence.

For clarity, only Propositions, Lemmas, Theorems, and Corollaries below are used as mathematical assertions with proofs.  Statements labelled as computational findings summarize finite searches and should be read as reproducible audit records rather than universal negative results.

\section{The Ap\'ery-type certificate criterion}

We begin by fixing the certificate standard used throughout the paper.  The formulation abstracts the decisive integer-linear-form feature of classical Ap\'ery-type arguments and their modern experimental variants \cite{Apery1979,Beukers1979,ChamberlandStraub2020,DoughertyBlissKoutschanZeilberger2022}.

\begin{definition}[Ap\'ery-type certificate for $e+\pi$]\label{def:apery-certificate}
A sequence $(A_n,B_n)\in\Z^2$ is an \emph{Ap\'ery-type certificate sequence} for $e+\pi$ if
\[
        L_n=A_n(e+\pi)+B_n
\]
satisfies
\[
        0<|L_n|\to0.
\]
\end{definition}

\begin{lemma}[Certificate lemma]\label{lem:certificate}
If there exists an Ap\'ery-type certificate sequence for $e+\pi$, then $e+\pi$ is irrational.
\end{lemma}

\begin{proof}
Suppose $e+\pi=a/b\in\Q$, with $a\in\Z$ and $b\in\Z_{>0}$.  Then
\[
        bL_n=A_na+bB_n\in\Z.
\]
Since $L_n\ne0$, the integer $bL_n$ is nonzero, so $|bL_n|\ge1$.  Hence $|L_n|\ge 1/b$ for every $n$, contradicting $L_n\to0$.
\end{proof}

\begin{remark}
This lemma is elementary, but it is the decisive arithmetic filter.  A rational approximation $-B_n/A_n\approx e+\pi$ is not sufficient unless the associated integer linear form $A_n(e+\pi)+B_n$ is itself small.
\end{remark}

\section{Factorial tail criteria for the rationality hypothesis}

This section proves exact eventual criteria for the hypothesis $e+\pi\in\Q$.  The criteria are useful because they identify the tail behavior that rationality would force; they are not, however, finite proofs of irrationality.

\subsection{The factorial tail of \texorpdfstring{$e$}{e}}

Let
\[
        E_N=\sum_{k=0}^N \frac1{k!},
        \qquad R_N=e-E_N,
        \qquad \rho_N=N!R_N.
\]
Then
\begin{equation}\label{eq:rho-expansion}
        \rho_N=\frac1{N+1}+\frac1{(N+1)(N+2)}
        +\frac1{(N+1)(N+2)(N+3)}+\cdots,
\end{equation}
and therefore
\begin{equation}\label{eq:rho-bound}
        0<\rho_N<1.
\end{equation}
Moreover,
\begin{equation}\label{eq:rho-recurrence}
        \rho_{N+1}=(N+1)\rho_N-1.
\end{equation}
This recurrence is the source of all three criteria below.

\subsection{A ceiling recurrence}

Define
\[
        A_N=\ceil{N!\pi}.
\]
Since $\pi$ is irrational \cite{Niven1947}, $N!\pi$ is never an integer, and hence
\[
        A_N-N!\pi\in(0,1).
\]

\begin{theorem}[Ceiling recurrence criterion]\label{thm:ceiling}
The following statements are equivalent:
\begin{enumerate}[label=\textup{(\roman*)}]
\item $e+\pi\in\Q$;
\item there exists $N_0$ such that
\begin{equation}\label{eq:ceiling-recurrence}
        A_{N+1}=(N+1)A_N-1
\end{equation}
for every $N\ge N_0$.
\end{enumerate}
\end{theorem}

\begin{proof}
Assume first that $e+\pi=r=a/b\in\Q$.  For $N\ge b$,
\[
        N!\pi=N!r-N!E_N-\rho_N.
\]
Both $N!r$ and $N!E_N$ are integers, while $0<\rho_N<1$.  Hence
\[
        A_N=N!r-N!E_N.
\]
Therefore
\[
\begin{aligned}
        A_{N+1}
        &=(N+1)!r-(N+1)!E_{N+1}  \\
        &=(N+1)N!r-(N+1)N!E_N-1 \\
        &=(N+1)A_N-1.
\end{aligned}
\]
This proves the eventual recurrence.

Conversely, suppose \eqref{eq:ceiling-recurrence} holds for all $N\ge N_0$.  Set
\[
        \delta_N=A_N-N!\pi.
\]
Then $0<\delta_N<1$, and the assumed recurrence gives
\[
        \delta_{N+1}=(N+1)\delta_N-1.
\]
By \eqref{eq:rho-recurrence}, the difference $\eta_N=\delta_N-\rho_N$ satisfies
\[
        \eta_{N+1}=(N+1)\eta_N.
\]
If $\eta_{N_0}\ne0$, then $|\eta_N|$ grows factorially, contradicting the boundedness of both $\delta_N$ and $\rho_N$.  Hence $\eta_{N_0}=0$, and consequently $\delta_N=\rho_N$ for all $N\ge N_0$.  Thus
\[
        A_N-N!\pi=N!(e-E_N),
\]
so
\[
        N!(e+\pi)=A_N+N!E_N\in\Z
\]
for all sufficiently large $N$.  This implies $e+\pi\in\Q$.
\end{proof}

\subsection{Factorial-Cantor digits}

Every $x\in[0,1)$ admits a factorial-Cantor expansion
\[
        x=\sum_{n=2}^\infty \frac{d_n}{n!},
        \qquad 0\le d_n\le n-1,
\]
with the usual ambiguity at eventually maximal tails.  Since $\pi$ is irrational, the fractional part of $\pi$ has a non-eventually-terminating expansion.  Write
\[
        \pi=3+\sum_{n=2}^\infty \frac{d_n(\pi)}{n!}.
\]

\begin{lemma}[Digit extraction]\label{lem:digit-extraction}
Let $x$ be irrational, and write its factorial-Cantor expansion using the non-eventually-maximal representative.  If $F_n=\floor{n!x}$, then, for every $n\ge2$,
\[
        d_n(x)=F_n-nF_{n-1}.
\]
\end{lemma}

\begin{proof}
It is enough to prove the statement for the fractional part of $x$, since the integer part cancels in $F_n-nF_{n-1}$.  Write $x=m+\sum_{k=2}^\infty d_k/k!$, with $m\in\Z$ and with no eventually maximal tail.  Then
\[
 n!x=n!m+\sum_{k=2}^{n} d_k\frac{n!}{k!}
      +\sum_{k=n+1}^{\infty} d_k\frac{n!}{k!}.
\]
The last sum lies in $[0,1)$: its supremum is obtained from the maximal digit tail and equals
\[
        \sum_{k=n+1}^{\infty}(k-1)\frac{n!}{k!}=1,
\]
but equality would force an eventually maximal tail.  Hence
\[
        F_n=n!m+\sum_{k=2}^{n} d_k\frac{n!}{k!}.
\]
Subtracting $nF_{n-1}$ leaves exactly $d_n$.
\end{proof}

\begin{theorem}[Factorial digit criterion]\label{thm:factorial-digit}
One has $e+\pi\in\Q$ if and only if
\begin{equation}\label{eq:digit-condition}
        d_n(\pi)=n-2
\end{equation}
for all sufficiently large $n$.
\end{theorem}

\begin{proof}
Let
\[
        F_n=\floor{n!\pi}.
\]
By Lemma \ref{lem:digit-extraction},
\[
        d_n(\pi)=F_n-nF_{n-1}.
\]
Moreover, $A_n=F_n+1$, since $n!\pi$ is not an integer.  The recurrence
\[
        A_n=nA_{n-1}-1
\]
is equivalent to
\[
        F_n+1=n(F_{n-1}+1)-1,
\]
and hence to
\[
        F_n=nF_{n-1}+n-2.
\]
By the digit-extraction formula, this is precisely $d_n(\pi)=n-2$.  Therefore the eventual recurrence of Theorem \ref{thm:ceiling} is equivalent to the eventual digit condition.
\end{proof}

The digit condition has a simple interpretation.  The factorial expansion of $e$ is
\[
        e=2+\frac1{2!}+\frac1{3!}+\frac1{4!}+\cdots.
\]
If $d_n(\pi)=n-2$ eventually, then the corresponding digits of $e+\pi$ become $n-1$, the maximal possible digit.  The maximal tail telescopes:
\[
        \sum_{n=N}^\infty \frac{n-1}{n!}
        =\sum_{n=N}^\infty\left(\frac1{(n-1)!}-\frac1{n!}\right)
        =\frac1{(N-1)!}.
\]
Thus an infinite tail becomes a rational carry.

\subsection{A divisibility platform}

Let
\[
        B_N=\floor{N!e}=\sum_{k=0}^N\frac{N!}{k!},
        \qquad Q_N=A_N+B_N.
\]
Then $Q_N/N!$ is a natural rational approximation to $e+\pi$.

\begin{theorem}[Platform divisibility criterion]\label{thm:platform}
One has $e+\pi\in\Q$ if and only if
\begin{equation}\label{eq:platform-divisibility}
        N\mid Q_N
\end{equation}
for all sufficiently large $N$.  Equivalently, the sequence $Q_N/N!$ is eventually constant.
\end{theorem}

\begin{proof}
If $e+\pi=r=a/b\in\Q$, then for $N\ge b$ the calculation in the proof of Theorem \ref{thm:ceiling} gives
\[
        A_N=N!r-B_N.
\]
Thus $Q_N=N!r$, so $Q_N/N!=r$ is eventually constant.  In particular $Q_N=NQ_{N-1}$ for sufficiently large $N$, and hence $N\mid Q_N$.

Conversely, suppose that $N\mid Q_N$ for all sufficiently large $N$.  Since
\[
        B_N=NB_{N-1}+1
\]
and
\[
        A_N-NA_{N-1}=d_N(\pi)-(N-1),
\]
we obtain
\[
        Q_N-NQ_{N-1}=d_N(\pi)-(N-2).
\]
For all sufficiently large $N$ we may assume $N\ge2$.  The right-hand side lies in the interval $[2-N,1]$.  If $N\mid Q_N$, then $N$ divides $Q_N-NQ_{N-1}$, and the only multiple of $N$ in that interval is $0$.  Hence $d_N(\pi)=N-2$ eventually.  The conclusion follows from Theorem \ref{thm:factorial-digit}.
\end{proof}

\begin{corollary}\label{cor:tail-not-finite}
Theorems \ref{thm:ceiling}, \ref{thm:factorial-digit}, and \ref{thm:platform} reduce the rationality hypothesis to eventual tail behavior.  To prove $e+\pi\notin\Q$ from these criteria alone, one would need to prove that the corresponding tail pattern fails infinitely often.
\end{corollary}

\begin{remark}
Corollary \ref{cor:tail-not-finite} explains why the preceding criteria are not yet proof mechanisms.  They are exact descriptions of what rationality would impose, but they do not provide a finite obstruction to rationality.
\end{remark}

\section{Mixed integral certificates}

The preceding section concerns tail behavior.  We now turn to finite constructions of integer linear forms.  The main identity in this section combines the elementary integral for $e$ with the arctangent integral for $\pi$.

For $P\in\Z[x]$, define
\begin{equation}\label{eq:SP-def}
        S_P(x)=\sum_{j\ge0}(-1)^jP^{(j)}(x),
        \qquad A(P)=S_P(1).
\end{equation}
Since $P$ is a polynomial, the sum is finite.

\begin{lemma}[Integration-by-parts identity]\label{lem:ibp}
For every polynomial $P$,
\begin{equation}\label{eq:ibp-identity}
        \int_0^1 P(x)e^x\dd x=A(P)e-S_P(0).
\end{equation}
\end{lemma}

\begin{proof}
Repeated integration by parts gives the antiderivative
\[
        \int P(x)e^x\dd x=e^x\sum_{j\ge0}(-1)^jP^{(j)}(x).
\]
Evaluating between $0$ and $1$ gives \eqref{eq:ibp-identity}.
\end{proof}

Since
\[
        \int_0^1\frac{\dd x}{1+x^2}=\frac\pi4,
\]
we obtain the following mixed certificate identity.

\begin{proposition}[Kernel identity]\label{prop:kernel-identity}
Let $P\in\Z[x]$ and define
\begin{equation}\label{eq:kernel-def}
        K_P(x)=P(x)e^x+\frac{4A(P)}{1+x^2}.
\end{equation}
Then
\begin{equation}\label{eq:kernel-identity}
        \int_0^1K_P(x)\dd x=A(P)(e+\pi)-S_P(0).
\end{equation}
In particular, the integral is an integer linear form in $e+\pi$.
\end{proposition}

\begin{proof}
By Lemma \ref{lem:ibp},
\[
        \int_0^1P(x)e^x\dd x=A(P)e-S_P(0).
\]
Also
\[
        \int_0^1\frac{4A(P)}{1+x^2}\dd x=A(P)\pi.
\]
Adding the two identities gives \eqref{eq:kernel-identity}.  Since $P\in\Z[x]$, both $A(P)$ and $S_P(0)$ are integers.
\end{proof}

A harmless derivative correction can adjust the integer part without changing the analytic nature of the problem.

\begin{corollary}[Derivative corrections]\label{cor:derivative-correction}
If $P,G\in\Z[x]$, then
\begin{equation}\label{eq:derivative-correction}
\int_0^1\left(P(x)e^x+\frac{4A(P)}{1+x^2}+G'(x)\right)\dd x
=A(P)(e+\pi)+G(1)-G(0)-S_P(0).
\end{equation}
\end{corollary}

\begin{proof}
Add the identity $\int_0^1G'(x)\dd x=G(1)-G(0)$ to Proposition \ref{prop:kernel-identity}.  Since $G\in\Z[x]$, the endpoint difference is an integer.
\end{proof}

\begin{definition}[Kernel certificate family]\label{def:kernel-certificate-family}
A sequence $P_n\in\Z[x]$ is called a \emph{kernel certificate family} if the associated forms
\[
        L(P_n)=A(P_n)(e+\pi)-S_{P_n}(0)
\]
are nonzero and satisfy $L(P_n)\to0$.
\end{definition}

The central question becomes whether one can force $K_{P_n}$ to be small on $[0,1]$ while keeping the arithmetic coefficients under control.

\section{Arithmetic obstruction to analytic fitting}

The identities above reduce the problem to a familiar tension in irrationality proofs.  Analytic constructions often produce small remainders before integerization, but irrationality requires small integer linear forms after all denominators and coefficient growth have been accounted for.  The following elementary lemma records the filter used in the later audit.

\begin{lemma}[Denominator-clearing filter]\label{lem:denominator-clearing}
Let
\[
        \Lambda_n=\alpha_n(e+\pi)+\beta_n,
        \qquad \alpha_n,\beta_n\in\Q,
\]
and let $D_n\in\Z_{>0}$ be such that $D_n\alpha_n,D_n\beta_n\in\Z$.  If the forms $\Lambda_n$ are to prove the irrationality of $e+\pi$ through Lemma \ref{lem:certificate}, then the integer forms
\[
        D_n\Lambda_n=(D_n\alpha_n)(e+\pi)+D_n\beta_n
\]
must be nonzero eventually and must satisfy $D_n|\Lambda_n|\to0$.
\end{lemma}

\begin{proof}
Lemma \ref{lem:certificate} applies only to integer linear forms in $e+\pi$.  Clearing denominators gives the displayed integer form.  Its absolute value is exactly $D_n|\Lambda_n|$.  Therefore the required smallness condition for the integer form is precisely $D_n|\Lambda_n|\to0$, together with eventual non-vanishing.

\end{proof}

\begin{corollary}[No-go filter for a proposed certificate mechanism]\label{cor:no-go-filter}
Let a proposed mechanism produce rational forms
\[
        \Lambda_n=\alpha_n(e+\pi)+\beta_n,\qquad \alpha_n,\beta_n\in\Q,
\]
with clearing factors $D_n$, or integer forms $L_n=A_n(e+\pi)+B_n$.  If, after denominator clearing and primitive reduction, the mechanism contains no nonzero subsequence of integer linear forms tending to zero, then it cannot prove $e+\pi\notin\Q$ by the Ap\'ery-type certificate criterion of Lemma~\ref{lem:certificate}.  In particular, a sequence whose strongest primitive outputs are merely continued-fraction shadows, with no independent degree-continuing non-CF improvement, is not a non-circular Ap\'ery-type certificate.
\end{corollary}

\begin{proof}
This is a direct application of Lemma~\ref{lem:denominator-clearing} and Definition~\ref{def:apery-certificate}.  The certificate lemma requires nonzero integer forms tending to zero.  If the proposed mechanism does not produce such forms after the necessary arithmetic reductions, it cannot supply the required contradiction to rationality.
\end{proof}

Thus the quantity that matters is not merely $|\Lambda_n|$, but the denominator-cleared size $D_n|\Lambda_n|$.  If the common denominator grows faster than the analytic remainder decays, the construction cannot yield an Ap\'ery-type certificate through these forms.

The same audit applies to integer-polynomial constructions.  When $P\in\Z[x]$, no denominator clearing is necessary in Proposition \ref{prop:kernel-identity}, but a candidate may still be nonprimitive: if
\[
        A(P)(e+\pi)-S_P(0)=g\bigl(A_0(e+\pi)-B_0\bigr),
        \qquad g=\gcd(A(P),S_P(0)),
\]
then the primitive form $A_0(e+\pi)-B_0$ is the relevant arithmetic object.  Apparent smallness caused by a common factor, or by rediscovering an ordinary continued-fraction approximant, is not a new Ap\'ery-type mechanism.  A successful construction must couple the analytic and arithmetic parts: the same structure that makes the remainder small must also enforce integrality, non-vanishing, and primitive smallness.

\section{Low-complexity certificate families}

This section records the outcome of several natural finite constructions, motivated in part by the experimental Ap\'ery-limit and Beukers-integral literature \cite{ChamberlandStraub2020,DoughertyBlissKoutschanZeilberger2022}.  The results in this section are not stated as impossibility theorems.  They are bounded computational findings and structural explanations for why the corresponding low-complexity families did not yield Definition \ref{def:apery-certificate}.

None of the observations below is used to prove a negative theorem about all possible Pad\'e, continued-fraction, holonomic, or lattice constructions.  Their purpose is to identify failure modes that any future certificate must pass: denominator clearing, primitive reduction, and continued-fraction-shadow auditing.

\subsection{Mixed Pad\'e approximation}

Using Machin's formula, a classical source of rapidly convergent formulae for $\pi$ \cite{BorweinBorwein1987},
\[
        \pi=16\arctan\frac15-4\arctan\frac1{239},
\]
set
\[
        T(z)=16\arctan\frac z5-4\arctan\frac z{239},
        \qquad S(z)=e^z+T(z).
\]
Then $S(1)=e+\pi$.  Direct Pad\'e approximation to $S$ gives rational functions whose values at $z=1$ approximate $e+\pi$ well.  However, after clearing denominators, the associated integer linear forms do not become small in the tested diagonal cases.  The observed behavior is consistent with the general obstruction in the previous section: Pad\'e approximation provides analytic closeness, but not an Ap\'ery-type integer certificate.

\subsection{Crossing separate approximations to \texorpdfstring{$e$ and $\pi$}{e and pi}}

Let $p_j/q_j$ approximate $e$, and let $P_k/Q_k$ approximate $\pi/4$ through a Lambert-type arctangent continued fraction.  Then
\begin{equation}\label{eq:crossed-form}
        Q_k(q_je-p_j)+q_j(Q_k\pi-4P_k)
        =q_jQ_k(e+\pi)-(Q_kp_j+4q_jP_k)
\end{equation}
is an integer linear form in $e+\pi$.  The difficulty is synchronization.  The good approximations to $e$ and to $\pi$ arise from different mechanisms, and their errors do not automatically decay on the same scale or with the sign needed to produce \eqref{eq:integer-linear-form-intro}.  In the tested ranges, crossed forms of the type \eqref{eq:crossed-form} did not yield a sequence tending to zero.

\subsection{Simple \texorpdfstring{$J$}{J}-fraction families}

Generalized continued fractions with low-degree polynomial coefficients, for example
\[
        b_n=un+v,
        \qquad a_n=cn^2
        \quad\text{or}\quad
        a_n=cn(n+r),
\]
were also considered.  Some choices produced limits numerically close to $e+\pi$; one family gave a value near
\[
        5.8598208827\ldots,
\]
whereas
\[
        e+\pi=5.8598744820\ldots.
\]
No identity with $e+\pi$ and no associated integer remainder estimate were found.  This illustrates a common risk: a search over simple continued fractions may find nearby constants rather than a certificate for the target constant.

\subsection{Holonomic ansatzes}

Several holonomic variants of Corollary \ref{cor:derivative-correction} were tested, including endpoint-Hermite interpolation, fixed-$P$ derivative corrections, Taylor cancellation, and integer Taylor-nullspace searches.  These methods can reduce local analytic residuals.  Nevertheless, in the rational-coefficient variants the denominator-cleared quantity $D_n|\Lambda_n|$ did not tend to zero, and in the integer variants the primitive forms did not organize into a decreasing nonzero sequence.  Thus ordinary holonomic fitting did not produce an Ap\'ery-type mechanism in the tested families.

\subsection{Rodrigues-type polynomial families}

A more structured class is given by Rodrigues-type polynomials
\begin{equation}\label{eq:rodrigues-family}
        P_n(x)=\frac1{n!}\frac{d^n}{dx^n}\left[x^n(1-x)^nR(x)^n\right],
\end{equation}
where $R\in\Z[x]$ is a low-degree polynomial.  The operator $(1/n!)d^n/dx^n$ preserves integrality of coefficients on integer polynomials, since it sends $x^m$ to $\binom{m}{n}x^{m-n}$.  Such families build in endpoint vanishing and factorial divisibility, which are features of many classical irrationality proofs.  However, for the tested $R$, the corresponding integer linear forms
\[
        A(P_n)(e+\pi)-S_{P_n}(0)
\]
did not tend to zero.  Endpoint divisibility alone is therefore not sufficient; it must be coupled to the mixed $e$--$\pi$ structure in a stronger way.

\section{Kernel-lattice search and continued-fraction shadows}

The most sensitive search used the kernel identity of Proposition \ref{prop:kernel-identity}.  The objective is to find $P\in\Z[x]$ such that
\[
        K_P(x)=P(x)e^x+\frac{4A(P)}{1+x^2}
\]
is small on $[0,1]$.  This is a function-space objective, not an explicit continued-fraction objective.  Nevertheless, the strongest candidates obtained by this method were governed by continued-fraction approximations after primitive reduction.

\subsection{Primitive reduction}

For a nonzero integer linear form
\[
        A\alpha-B,
        \qquad A,B\in\Z,
        \qquad \alpha=e+\pi,
\]
let $g=\gcd(A,B)$ and write
\[
        A_0=A/g,
        \qquad B_0=B/g.
\]
The primitive form is $A_0\alpha-B_0$.  Since
\[
        A\alpha-B=g(A_0\alpha-B_0),
\]
any apparent smallness of the nonprimitive form must be judged after reduction.

\begin{definition}[Continued-fraction shadow]\label{def:cf-shadow}
Let $\alpha=e+\pi$.  A primitive form $A\alpha-B$ is called a \emph{continued-fraction shadow}, or \emph{CF-shadow}, if $B/A$ is a convergent or an intermediate convergent of the continued fraction of $\alpha$.
\end{definition}

CF-shadows are not useless approximations, but they do not by themselves constitute a new Ap\'ery-type mechanism.  They may simply repackage the classical best-approximation structure of the real number $e+\pi$.

\subsection{A degree-\texorpdfstring{$14$}{14} candidate}\label{subsec:degree14}

One degree-$14$ polynomial found by the kernel-lattice search gave
\[
        A(P)=245223,
        \qquad S_P(0)=1436976.
\]
Therefore
\[
        L(P)=245223(e+\pi)-1436976
        \approx 1.11462317066\times10^{-4}.
\]
The sampled kernel was uniformly small on the search grid, with approximate range
\[
        -0.00121336\le K_P(x)\le0.00131354.
\]
At first sight this is a strong mixed-kernel signal.  However,
\[
        245223=9\cdot27247,
        \qquad 1436976=9\cdot159664,
\]
so
\[
        245223(e+\pi)-1436976
        =9\bigl(27247(e+\pi)-159664\bigr).
\]
The rational number
\[
        \frac{159664}{27247}
\]
is a continued-fraction convergent of $e+\pi$.  One way to certify this assertion is to compute the continued-fraction cylinder
\[
        [5;1,6,7,3,21,2,1,2]=\frac{159664}{27247},
        \qquad
        \frac{159664+59759}{27247+10198}=\frac{219423}{37445},
\]
and to verify, by rigorous interval bounds for $e$ and $\pi$, that
\[
        \frac{159664}{27247}<e+\pi<\frac{219423}{37445}.
\]
Thus the candidate is a CF-shadow after primitive reduction.

\subsection{Structured-basis search}

The final search examined several structured bases: Bernstein, shifted Legendre, Rodrigues-derivative, symmetric, endpoint-monomial, and endpoint-Legendre bases.  Each candidate was primitive-reduced and tested against the continued-fraction convergents and intermediate convergents of $e+\pi$.

\begin{finding}\label{find:kernel-lattice}
In the bounded structured-basis searches described by the protocol in Appendix \ref{app:kernel-lattice}, the implemented runs found many non-CF primitive candidates, but no basis family produced simultaneous degree-continuation of both the kernel norm and the primitive error.  The best primitive error returned to the CF-shadow
\[
        27247(e+\pi)-159664.
\]

\end{finding}

The final audit retained $145$ raw candidate records and $133$ classified primitive records.  Figure~\ref{fig:round8-degree-audit} pulls this bookkeeping out of Appendix~\ref{app:data-summary} and displays the main experimental obstruction.  The left panel plots the primitive residuals
\[
        \log_{10}|A_0(e+\pi)-B_0|
\]
against the polynomial degree after primitive reduction and CF-shadow classification.  The right panel restricts to the best non-CF candidate at each degree and compares residual precision with denominator size.  The visual pattern is the negative result: the non-CF denominator size grows, but the residual precision does not organize into an improving degree-continuing family.

\begin{figure}[t]
\centering
\includegraphics[width=0.96\textwidth]{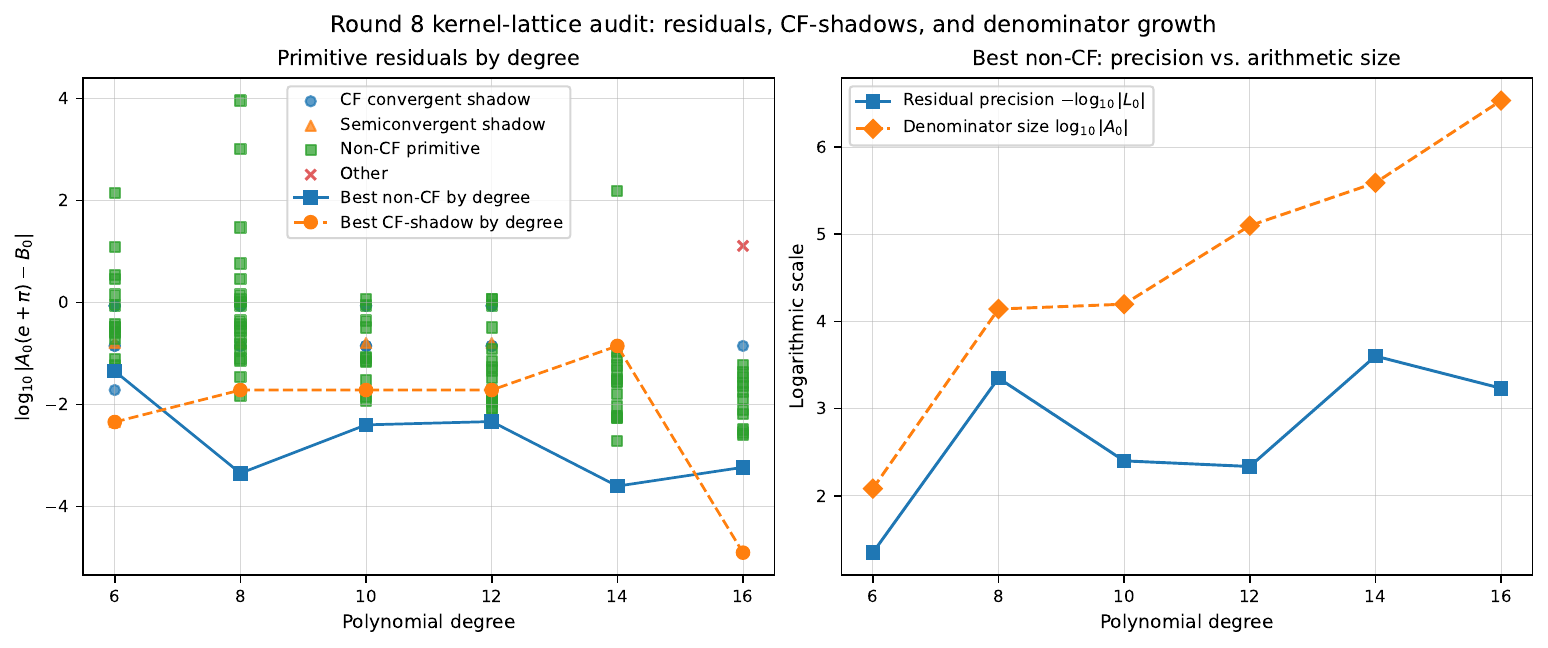}
\caption{Round~8 kernel-lattice audit.  Left: primitive residuals by polynomial degree, separated into CF-convergent shadows, semiconvergent shadows, and non-CF primitive candidates.  The lowest residuals are dominated by CF-shadows.  Right: for the best non-CF candidate at each degree, the residual precision $-\log_{10}|A_0(e+\pi)-B_0|$ is compared with the arithmetic size $\log_{10}|A_0|$.  The denominator grows with degree, but the non-CF precision does not show a monotone Ap\'ery-type improvement.}
\label{fig:round8-degree-audit}
\end{figure}

Examples of non-CF candidates included
\[
        37602(e+\pi)-220343
\]
and
\[
        567898(e+\pi)-3327811.
\]
These candidates did not form a stable improving family as the degree increased.

\begin{remark}
The kernel-lattice method remains mathematically useful as a diagnostic.  Its failure mode is informative: even a function-space objective can recover continued-fraction shadows of $e+\pi$.  Any future kernel-based certificate must therefore pass primitive reduction and CF-shadow testing before it can be regarded as a genuinely new mechanism.
\end{remark}

\section{Comparison of mechanisms}

The following table summarizes the mechanisms considered above and their observed obstruction.

\begin{center}
\renewcommand{\arraystretch}{1.18}
\begin{tabular}{L{0.34\textwidth}L{0.55\textwidth}}
\toprule
\textbf{Mechanism} & \textbf{Observed obstruction} \\
\midrule
Single rational approximation & Can reduce to ordinary continued-fraction fitting. \\
Factorial-Cantor tail criteria & Exact equivalences, but require ruling out eventual tail behavior. \\
Mixed Pad\'e approximation & Analytic approximation improves, but integer linear forms grow after clearing denominators. \\
Crossed $e$ and $\pi$ mechanisms & Separate errors do not synchronize into a small mixed form. \\
Simple $J$-fractions & Searches find nearby constants rather than identities with $e+\pi$. \\
Holonomic ansatzes & Local cancellation does not survive denominator and coefficient growth. \\
Rodrigues-type families & Endpoint divisibility alone does not force mixed Ap\'ery cancellation. \\
Kernel-lattice search & Strongest primitive forms are dominated by CF-shadows. \\
Non-CF degree-continuing families & Not found in the tested structured bases. \\
\bottomrule
\end{tabular}
\end{center}

The common obstruction can be summarized as follows:
\[
        \text{analytic smallness}\quad\not\Rightarrow\quad
        \text{arithmetic smallness}.
\]
A successful Ap\'ery-type proof must combine four properties in the same structure:
\begin{enumerate}[label=\textup{(\alph*)}]
\item integer coefficients;
\item a nonzero remainder;
\item a remainder estimate tending to zero;
\item denominator and coefficient growth controlled strongly enough not to destroy the estimate.
\end{enumerate}
The tested constructions usually achieve one or two of these properties, but not all four simultaneously.

\section{Conclusion}

This paper establishes several exact tail criteria for the rationality of $e+\pi$ and formulates a finite Ap\'ery-type certificate framework for attacking its irrationality.  The kernel identity
\[
        \int_0^1\left(P(x)e^x+\frac{4A(P)}{1+x^2}\right)\dd x
        =A(P)(e+\pi)-S_P(0)
\]
provides a natural mixed source of integer linear forms.  However, the low-complexity mechanisms tested here do not produce a nonzero sequence tending to zero.  Their shared failure is that analytic cancellation and arithmetic integrality are not generated by the same structure.

The resulting lesson is precise.  A proof of $e+\pi\notin\Q$ by Ap\'ery-type means cannot rely on approximation quality alone.  It must produce integer linear forms whose smallness survives denominator clearing, coefficient growth, primitive reduction, and continued-fraction auditing.  Within the tested families, no such non-circular mechanism was found.

\appendix

\section{CF-shadow audit protocol}\label{app:cf-shadow}

For a candidate integer linear form $A\alpha-B$, with $\alpha=e+\pi$:
\begin{enumerate}
\item Compute $g=\gcd(A,B)$.
\item Set $A_0=A/g$ and $B_0=B/g$.
\item Compute the continued fraction of $\alpha$ to denominator range exceeding $A_0$.
\item Check whether $B_0/A_0$ is a convergent.
\item If not, check whether $B_0/A_0$ is an intermediate convergent.
\item If either test is positive, label the candidate as a CF-shadow.
\item Otherwise retain it as a non-CF primitive candidate and test degree-continuation.
\end{enumerate}

\section{Kernel-lattice search protocol}\label{app:kernel-lattice}

For a structured polynomial basis $\{\phi_j\}_{j=0}^d$:
\begin{enumerate}
\item Write $P(x)=\sum_{j=0}^d c_j\phi_j(x)$ with $c_j\in\Z$.
\item Compute $A(P)=\sum_{r\ge0}(-1)^rP^{(r)}(1)$.
\item Form $K_P(x)=P(x)e^x+4A(P)/(1+x^2)$.
\item Sample $K_P$ on a grid in $[0,1]$ and build a lattice objective for small sampled values.
\item Use lattice reduction to obtain candidate integer coefficient vectors $c$.
\item For each candidate, compute the exact integer linear form
\[
        L(P)=A(P)(e+\pi)-S_P(0).
\]
\item Primitive-reduce $A(P)$ and $S_P(0)$.
\item Apply the CF-shadow audit from Appendix \ref{app:cf-shadow}.
\item Retain only non-CF candidates and examine degree-continuation.
\end{enumerate}

\section{Data summary from the final kernel-lattice search}\label{app:data-summary}

The final structured-basis search retained $145$ raw candidate records before primitive reduction and classification.  In the classified primitive records, $32$ were CF-convergent shadows, $5$ were semiconvergent shadows, and $95$ were classified as non-CF.  The best primitive error still came from the CF-shadow
\[
        27247(e+\pi)-159664.
\]
The best non-CF candidates did not organize into a degree-continuing family and did not show simultaneous improvement of kernel norm and primitive error.  Figure~
ef{fig:round8-degree-audit} visualizes the degree-wise residual and denominator data extracted from this final audit.  These counts are bookkeeping data from the finite search and are not used in the proofs of the preceding theorems.

\section{Certification of the degree-14 CF-shadow}\label{app:cf-certification}

For completeness we spell out the continued-fraction check used for the degree-$14$ candidate in Section \ref{subsec:degree14}.  The convergents associated with the prefix
\[
        [5;1,6,7,3,21,2,1,2]
\]
include
\[
        \frac{59759}{10198}
        \quad\text{and}\quad
        \frac{159664}{27247}.
\]
With indices starting at the initial partial quotient $a_0=5$, this is the even-index convergent $p_8/q_8$.  The adjacent endpoint of the corresponding continued-fraction cylinder is
\[
        \frac{159664+59759}{27247+10198}=\frac{219423}{37445}.
\]
Therefore the inequalities
\[
        \frac{159664}{27247}<e+\pi<\frac{219423}{37445}
\]
certify that $159664/27247$ is a continued-fraction convergent of $e+\pi$.  Such inequalities can be verified rigorously using the factorial-tail bound for $e$ and any standard Machin-series interval bound for $\pi$; no unproved number-theoretic input is involved.

\end{document}